\documentclass[11pt]{article}
\usepackage{amsmath}
\usepackage{amssymb}
\usepackage{amsthm}
\usepackage{calc}
\usepackage{enumerate}
\font\goth=eufm10
\newcommand{\gc}{\hbox{\goth c}}
\newcommand{\emp}{\emptyset}

\newcommand{\ben}{\mathbb N}
\newcommand{\ber}{\mathbb R}
\newcommand{\bec}{\mathbb C}
\newcommand{\bed}{\mathbb D}
\newcommand{\beg}{\mathbb G}
\newcommand{\beq}{\mathbb Q}
\newcommand{\bez}{\mathbb Z}

\newcommand{\nhat}[1]{\{1,2,\ldots,#1\}}

\newcommand{\pf}{{\mathcal P}_f}

\newcommand{\defn}{\emph}
\newcommand{\vvdots}[1]{\mathrel{\makebox[\widthof{\ensuremath{#1}}]{\vdots}}}
\newcommand{\vvvdots}[1]{\makebox[\widthof{\ensuremath{#1}}]{\vdots}}
\newcommand{\invis}[1]{\mathrel{\makebox[\widthof{\ensuremath{#1}}]{ }}}
\newcommand{\hslemmas}[1]{Lemma~\ref{hslemmas}({\ref{#1}})}
\newtheorem{theorem}{Theorem}[section]

\newtheorem{lemma}[theorem]{Lemma}
\newtheorem{proposition}[theorem]{Proposition}
\newtheorem{question}[theorem]{Question}
\theoremstyle{definition}
\newtheorem{definition}[theorem]{Definition}
\title{Distinguishing subgroups of the rationals by their Ramsey properties}
\author{Ben Barber
        \footnote{School of Mathematics,
                  University of Birmingham,
                  Edgbaston,
                  Birmingham, B15 2TT, UK.\hfill\break
                  {\tt b.a.barber@bham.ac.uk}}
        \and
        Neil Hindman
        \footnote{Department of Mathematics,
                  Howard University,
                  Washington, DC 20059, USA.\hfill\break
                  {\tt nhindman@aol.com}}
        \thanks{This author acknowledges support received from the National
                Science Foundation (USA) via Grant DMS-1160566.}
        \and
        Imre Leader
        \footnote{Department of Pure Mathematics and Mathematical Statistics,
                  Centre for Mathematical Sciences,
                  Wilberforce Road, Cambridge, CB3 0WB, UK.\hfill\break
                  {\tt i.leader@dpmms.cam.ac.uk}}
        \and
        Dona Strauss
        \footnote{Department of Pure Mathematics,
                  University of Leeds, Leeds LS2 9J2, UK.\hfill\break
                  {\tt d.strauss@hull.ac.uk}}                 
}

\begin{document}

\maketitle
\begin{abstract}
A system of linear equations with integer coefficients is 
{\it partition regular\/}
over a subset $S$ of the reals if, whenever $S\setminus\{0\}$ is finitely 
coloured, there is a solution to the system contained in one colour class. 
It has been known for some time that there is an infinite system
of linear equations that is partition regular over $\ber$ but not over $\beq$,
and it was recently shown (answering a long-standing open question) that one
can also distinguish $\beq$ from $\bez$ in this way.

Our aim is to show that the transition from $\bez$ to $\beq$ is not
sharp: there is an infinite chain of subgroups of $\beq$, each of which has a 
system 
that is partition regular over it but not over its predecessors. We actually
prove something stronger: our main result is that if $R$ and $S$ are subrings 
of $\beq$ with
$R$ not contained in $S$, then there is a system that is partition regular
over $R$ but not over $S$. This implies, for example, that the chain above may 
be taken to be uncountable.
\end{abstract}

\section{Introduction}

Consider a system of $u$ linear equations in $v$ unknowns.
\begin{align*}
 a_{1,1} x_1 + a_{1,2} x_2 + \cdots + a_{1,v} x_v &= 0 \\
 a_{2,1} x_1 + a_{2,2} x_2 + \cdots + a_{2,v} x_v &= 0 \\
 \vvvdots{a_{2,1} x_1} \invis{+} \vvvdots{a_{2,2} x_2} \invis{+}\! \ddots\! \invis{+} \vvvdots{a_{2,v} x_v} & \invis{=} \vvvdots{0} \\
 a_{u,1} x_1 + a_{u,2} x_2 + \cdots + a_{u,v} x_v &= 0 \\
\end{align*}
If the coefficients are rational
numbers and the set $\ben$ of positive integers is finitely coloured,
can one be guaranteed to find monochromatic $x_1,x_2,\ldots,x_v$
solving the given system? That is, is the system of equations
{\it partition regular\/}? In \cite{R}, Rado answered this 
question, showing that the system is partition regular if and only 
if the matrix of coefficients 
\[
A = \begin{pmatrix}
        a_{1,1} & a_{1,2} & \cdots & a_{1,v} \\
        a_{2,1} & a_{2,2} & \cdots & a_{2,v} \\
        \vdots  & \vdots  & \ddots & \vdots  \\
        a_{u,1} & a_{u,2} & \cdots & a_{u,v}
    \end{pmatrix}
\]
satisfies the \emph{columns condition}:

\begin{definition} \label{defcc}
Let $u,v\in\ben$ 
and let $A$ be a $u\times v$ matrix with entries 
from $\bez$.  Denote the columns of $A$ by $\langle\vec c_i\rangle_{i=1}^{v}$.
  The matrix $A$ satisfies the {\it columns condition \/} if
there exist $m\in\nhat{v}$ and a partition $\langle I_t\rangle_{t=1}^{m}$ of 
$\nhat{v}$ such that
\begin{enumerate}[(1)]
\item $\sum_{i\in I_1}\vec c_i=\vec 0$;
\item for each $t\in \{2,3,\ldots,m\}$, if any, 
$\sum_{i\in I_t}\vec c_i$ is a linear combination with coefficients
from $\beq$ of $\{\vec c_i:i\in\bigcup_{j=1}^{t-1}I_j\}$.
\end{enumerate}
\end{definition}

If one considers the same equations over $\ber$, an easy compactness argument
shows that a finite system of equations is partition regular
over the reals if and only if it is partition regular over the integers.

Note that the restriction to integer coefficients might as
well be to rational coefficients, as we are always free to multiply each
equation by a constant. We remark in passing that if one were to allow
coefficients that are not rational, then the situation for finite systems
is again understood: in \cite{Rb}, Rado extended his result by showing that if 
$R$ is any subring
of the set $\bec$ of complex numbers and the entries of $A$ are from $R$,
then the system of equations is partition regular over $R$ if and only if
the matrix $A$ satisfies the columns condition over the field $F$ generated by $R$
(which means that we replace `linear combination with coefficients from
$\beq$' by `linear combination with coefficients from $F$').

So the partition regularity of finite systems is quite settled.  The case
with {\it infinite} systems of linear equations, however, is much
harder, and in general is still poorly understood. There is by now a large 
literature on the subject (see the survey \cite{H}), 
but there is nothing resembling a characterisation of those infinite 
systems that are partition regular over $\bez$, $\beq$, or any other interesting subset of $\bec$.  

One difference between finite and infinite systems is the main focus of this
paper. As stated above, if a finite system of linear equations has rational coefficients,
it is a consequence of Rado's original theorems that the system is
partition regular over $\ben$ if and only if it is partition regular over $\ber$ (and thus
if and only if it is partition regular over $\bez$ or over $\beq$).

It was shown in \cite{HLS} that the infinite system of equations 
$y_n=x_n-x_{n+1}$ ($n = 0,1,2,\ldots$) is partition regular over $\ber$ but not over $\beq$.  It 
was an open problem for some time whether every system of 
linear equations with rational coefficients that is partition regular
over $\beq$ must also be partition regular over $\ben$.  (We remark in passing
that there is no difference between $\ben$ and $\bez$ in this regard, because
if a system has a bad $k$-colouring over $\ben$ then it also has a bad 
$2k$-colouring over $\bez$, obtained by copying the colouring of $\ben$ to
the negative integers but using $k$ new colours---so we switch freely between
$\ben$ and $\bez$ in this paper.) 

This question was
answered in the negative in \cite[Theorem~12]{BHL} by showing that the
following system of equations is partition regular over $\bed$, the set
of dyadic rationals. (It is not partition regular over $\ben$ because it
has no solutions in $\ben$ at all.)
\begin{align*}   
                    x_{1,1} + 2^{-1}y & = z_{1,1} \\
          x_{2,1} + x_{2,2} + 2^{-2}y & = z_{2,1} + z_{2,2} \\
                                      & \vvdots{=} \\
 x_{n,1} + \cdots + x_{n,n} + 2^{-n}y & = z_{n,1} + \cdots + z_{n,n}\\
                                      & \vvdots{=}
\end{align*}

In this paper we extend this result by considering the following system of
equations, which is a generalisation of another system introduced in 
\cite{BHL}.  Let $\alpha\in\ben$
and, for $n \geq 2$ and $1 \leq i \leq \alpha$, let $d_{n,i}$
be an element of some infinite ring $R$. (We take rings to have identities.)

\medskip\noindent
{\bf System~$(*)$}:
\vskip -20 pt

\begin{align*}
 x_{2,1} + x_{2,2}+d_{2,1}y_1+d_{2,2}y_2+\cdots+d_{2,\alpha}y_\alpha & = z_2 \\
 x_{3,1} + x_{3,2} + x_{3,3}+d_{3,1}y_1+d_{3,2}y_2+\cdots+d_{3,\alpha}y_\alpha & = z_3 \\
                                      & \vvdots{=} \\
x_{n,1} + \cdots + x_{n,n}+d_{n,1}y_1+d_{n,2}y_2+\cdots+d_{n,\alpha}y_\alpha& = z_n\\
& \vvdots{=}
\end{align*}

In Section~\ref{general} we prove (Theorem \ref{dspr}) that, 
if $R$ satisfies a certain
technical condition, then System $(*)$ is partition regular over $R$.  (This 
technical condition is satisfied by all subrings of $\beq$.)
We actually show that System~$(*)$ satisfies a slightly stronger condition:
it is strongly partition regular over $R$.  

\begin{definition} \label{defspr} 
Let $R$ be a ring. A system of linear 
equations (with coefficients in $R$) is \defn{strongly partition regular over $R$} if, whenever
$R$ is finitely coloured, there exists a monochromatic solution to the system
with distinct variables taking on different values.
\end{definition}

[This is the reason for starting System~$(*)$ at $n=2$.  If we include the equation 
for $n=1$, then the system remains partition regular, but we cannot ensure that $x_{1,1}$ 
and $z_1$ receive different colours: consider the case where $d_{1,1} = d_{1,2} = \cdots d_{1,\alpha} = 0$.]

In Section~\ref{applications} we apply the results of Section~\ref{general} to show that 
there is an infinite increasing sequence $\langle G_n\rangle_{n=1}^\infty$ of 
subgroups of $\beq$ with the property that, for each $n$, there is a 
choice of the coefficients $\langle d_{n,i}\rangle_{n=1}^\infty$ 
making System~$(*)$ strongly partition
regular over $G_{n+1}$ while it is not partition regular over 
$G_n$. We actually prove rather more (Theorem \ref{onenotother}): 
this separation property holds for any
two subrings of the rationals. This means that, for example, there is even an
uncountable chain with this property. We close with some open problems.

The results of Section~\ref{general} make substantial use of the algebraic structure of the
Stone--\v Cech compactification of a discrete semigroup, which we 
briefly introduce in Section~\ref{betaS}.

\section{The Stone--\v Cech compactification} \label{betaS}

Let $S$ be a semigroup.  We shall be concerned here exclusively with
commutative semigroups, so we will denote the operation of $S$ by $+$.  For
proofs of the assertions made here, see the first five chapters of \cite{HS}.

The \emph{Stone--\v Cech compactification} of $S$ is denoted by $\beta S$.  The points of $\beta S$ are the ultrafilters on $S$.  We identify the principal
ultrafilters with the points of $S$, whereby we pretend that $S \subseteq \beta
S$.
The operation on $S$ extends to an operation on $\beta S$, also
denoted by $+$, with the property that, for $x \in S$ and $q \in \beta S$, the functions
\begin{align*}
 p & \mapsto x + p \qquad \text{ and} \\
 q & \mapsto p + q
\end{align*}
are continuous.
(The reader should be cautioned that $(\beta S,+)$ is almost 
certain to be non-commutative: the centre of $(\beta\ben,+)$ is
$\ben$.) Given $A\subseteq S$ and $p,q\in\beta S$, $A\in p+q$
if and only if $\{x\in S:-x+A\in q\}\in p$, where $-x+A=\{y\in S:x+y\in A\}$.

With this operation, $(\beta S,+)$ is a compact Hausdorff right topological
semigroup.  Any such object contains idempotents, points $p$ such that $p = p +
p$.  The semigroup $\beta S$ has a smallest two-sided ideal, $K(\beta S)$,
which is the union of all of the minimal right ideals of $\beta S$ as well as
the union of all of the minimal left ideals of $\beta S$.  The intersection of
any minimal right ideal with any minimal left ideal is a group (and any two
such groups are isomorphic).  In particular, there are idempotents in $K(\beta
S)$---such idempotents are called {\it minimal\/}. 

A subset $A$ of $S$ is
\begin{itemize}
\item an \emph{IP-set} if it is a member of some idempotent;
\item \emph{central} if it is a member of some minimal idempotent;
\item \emph{central*} if it is a member of every minimal idempotent;
\item an \emph{IP*-set} if it is a member of every idempotent.
\end{itemize}
Equivalently, $A$ is an IP*-set if, whenever $\langle x_n\rangle _{n=1}^\infty$ is
a sequence in $S$, there exists $F\in \pf(\ben)$ such that $\sum_{n\in
F}x_n\in A$.  We will use this to show that certain sets are central*.

We will also require the following more specialised results from \cite{HS}.

\begin{lemma} \label{hslemmas}
\begin{enumerate}[{\rm (a)}]
\item \label{non-principal}
Let $G$ be a commutative group.  Then every minimal idempotent in $\beta G$ is non-principal.
\item \label{star}
Let $S$ be a semigroup,
let $p$ be an idempotent in $\beta S$ and, for $C \in p$, let $C^\star = \{s \in S : -s + C \in p \}$.  Then $C^\star \in p$ and, for each $s \in C^\star$, 
we have $-s + C^\star \in p$.
\end{enumerate}
\end{lemma}

\begin{proof}
(a) Let $p$ be a principal ultrafilter.  Then $p$ is idempotent if and only if $p+p=p$ in $G$, 
that is, if and only if $p=0$.  Suppose that $0$ were a minimal idempotent.  Then by Theorem~1.48 of \cite{HS}, 
$\beta S + 0 = \beta S$ is a minimal left ideal.  But by Corollary~4.33 of \cite{HS}, 
$\beta S \setminus S$ is a left ideal, contradicting the minimality of $\beta S$.

(b) Lemma 4.14 (and preceding discussion) of \cite{HS}.
\end{proof}

\section{General results}\label{general}

In this section we will show that System $(*)$, with coefficients $d_{n,i}$ in some infinite ring $R$, is
strongly partition regular over $R$.  In fact, we shall establish a stronger 
conclusion.

\begin{definition} \label{centrallypr} 
Let $(S,+)$ be a semigroup.  
\begin{itemize}
\item A system of linear equations is \defn{centrally partition regular over $S$} if, 
whenever $A$ is a central subset of $S$, there exists a solution to the system contained in $A$.

\item A system of linear equations is \defn{strongly centrally partition regular over $S$} if, 
whenever $A$ is a central subset of $S$, there exists a solution to the system contained in $A$
with distinct variables taking on different values.
\end{itemize}
\end{definition}

Notice that, since whenever a semigroup is finitely coloured, one colour class must be 
central, it follows that, if a system of equations is strongly centrally partition regular, then it is
strongly partition regular.

We use the usual additive notation
\begin{align*}
 A + B & = \{a + b : a \in A, b \in B\} \\
 A - B & = \{a - b : a \in A, b \in B\} \\
    kA & = A + \cdots + A \quad (k \text{ times})
\end{align*}
and write
$k\cdot A=\{k\cdot a:a\in A\}$.

We shall need the following result from \cite{BHLS}.

\begin{lemma} \label{mcdotg}  
Let $(G,+)$ be a commutative group and assume that
$c\cdot G$ is a central* set for each $c\in\ben$.  Let $C$ be a 
central subset of $G$.  Then there is an $m \in \ben$ and a $k$ such that,
if $n\geq k$, then $m\cdot G\subseteq C-nC$.
\end{lemma}

\begin{proof}
\cite[Lemma 3.7]{BHLS}.
\end{proof}


\begin{definition}
Let $A$ be a $u\times v$ matrix with entries from a ring $R$.  
An element $a_{i,j}$ of $A$ is a {\it first entry\/} of $A$ if $a_{i,k} = 0$ for $k < j$ 
and $a_{i,j} \neq 0$.  We say that $A$ satisfies the \emph{weak first entries condition} 
if no row of $A$ is $\vec 0$ and if $a_{i,k}$ and $a_{j,k}$ are first entries of $A$, then $a_{i,k}=a_{j,k}$.
\end{definition}

We call this the \emph{weak} first entries condition because, as usually defined
with $R=\beq$, one assumes that first entries are positive---which of course
does not make sense for general rings.


\begin{lemma} \label{FECimpliesCIPR}
Let $R$ be an infinite ring.  Let $u, v \in \ben$, let $A$ be a $u \times v$ matrix with entries from $R$ that satisfies the weak first entries condition, and suppose that $c \cdot R$ is central* in $R$ for each first entry $c$ of $A$.  Let $C$ be central in $R$.  Then there is an $\vec x$ in $(R\setminus \{0\})^v$ such that $A \vec x \in C^u$.
\end{lemma}

\begin{proof}
This is a special case of \cite[Theorem~15.5]{HS}.  That theorem was stated only for coefficients 
which were natural numbers so that it made sense in an arbitrary semigroup, but the proof in the case of rings is 
nearly identical. 
\end{proof}


\begin{theorem}\label{dpr} 
Let $R$ be an infinite ring and assume
that, for each $m\in\ben$, $m\cdot R$ is central* in $R$.
Let $\alpha\in\ben$
and, for each $n \geq 2$ and $1 \leq i \leq \alpha$, let $d_{n,i}$
be in $R$.  Then for each central subset $C$ of $R$ there is a solution
\[
y_1,y_2,\ldots,y_\alpha,x_{2,1},x_{2,2},z_2,x_{3,1},x_{3,2},x_{3,3},z_3,\ldots
\]
of System $(*)$
contained in $C$; that is, System $(*)$  is centrally partition regular over $R$.
Moreover, the solution
can be chosen so that $y_1$, $y_2$, \ldots, $y_\alpha$ are distinct.
\end{theorem}

\begin{proof} 
Let $C$ be central in $R$.  There is an idempotent $p\in \beta S$ such that $C \in p$, and by \hslemmas{non-principal}, $p \neq 0$.  Hence $C \setminus \{0\} \in p$, so $C \setminus \{0\}$ is also central and we may assume that $0 \notin C$.

By Lemma \ref{mcdotg}, there is an $m \in \ben$ and a $k$ such that,   
if $n\geq k$, then $m\cdot G\subseteq C-nC$.  Since $m \cdot G$ is central*, $C \cap m \cdot G$ is central.

Let $b_1 = 0$ and, for $2 \leq j \leq k$, let $b_j=b_{j-1} + j$.
Let $v=b_k+\alpha$ and
let $A$ be the $(k-1)\times v$ matrix with entries given by
\[
a_{i,j}=\begin{cases}
 1 & \text{if }b_i < j \leq b_{i+1};\\
 d_{i+1,t} & \text{if } j = b_k + t;\\
 0 & \text{otherwise.}
\end{cases}
\]
Let $B$ be an $\binom \alpha 2 \times v$ matrix such that for every 
$b_k < i < j \leq b_k + \alpha$, some row of $B$ has a $1$ in position $i$ and a $-1$ in position $j$, 
with all other entries equal to $0$. (If $\alpha=1$, let $B$ be empty.)
Thus, for example, if $k=4$ and $\alpha = 3$, then 
\[
\begin{pmatrix}
A \\ B
\end{pmatrix} =
\begin{pmatrix}
1&1&0&0&0&0&0&0&0&d_{2,1}&d_{2,2}&d_{2,3}\\
0&0&1&1&1&0&0&0&0&d_{3,1}&d_{3,2}&d_{3,3}\\
0&0&0&0&0&1&1&1&1&d_{4,1}&d_{4,2}&d_{4,3}\\
0&0&0&0&0&0&0&0&0&1&-1&0 \\
0&0&0&0&0&0&0&0&0&1&0&-1 \\
0&0&0&0&0&0&0&0&0&0&1&-1 \\
\end{pmatrix}.
\]
Let $I$ be the $v\times v$ identity matrix.
Then $\begin{pmatrix}I\\ A \\ B\end{pmatrix}$
satisfies the first entries condition with each first entry equal to $1$, so by Theorem~\ref{FECimpliesCIPR} there exist
\[
 x_{2,1},x_{2,2},x_{3,1},x_{3,2},x_{3,3},\ldots,x_{k,1},x_{k,2},\ldots,x_{k,k},y_1,y_2,\ldots,y_\alpha
\] 
such that all entries of 
\[
\begin{pmatrix}I\\ A\\ B\end{pmatrix}
\begin{pmatrix}x_{2,1}\\
\vdots\\ x_{k,k}\\ y_1\\
\vdots\\
y_\alpha
\end{pmatrix}
\]
are in $C\cap m\cdot G$.

For $2 \leq n \leq k$, let $z_n =  x_{n,1} + \cdots x_{n,n} + d_{n,1}y_1 + \cdots + d_{n,\alpha} y_\alpha$.  
The submatrix $I$ ensures that the $x_{n,j}$ ($2 \leq n \leq k$ and $1 \leq j \leq n$) and $y_i$ 
($1 \leq i \leq \alpha$) are in $C$, the submatrix $A$ ensures that the $z_n$ ($2 \leq n \leq k$) are in $C$, 
and the submatrix $B$ ensures that $y_i \neq y_j$ ($1 \leq i < j \leq \alpha$) as $y_i - y_j \in C \subseteq R \setminus \{0\}$.

For $n > k$, $d_{n,1}y_1 + \cdots + d_{n,\alpha}y_\alpha \in m\cdot G \subseteq C-nC$, so choose $z_n$ and $x_{n,1} \ldots x_{n,n}$ in $C$ such that $d_{n,1}y_1 + \cdots + d_{n,\alpha}y_\alpha = z_n - x_{n,1}-\cdots-x_{n,n}$.
\end{proof}

We are now ready for the main result of this section.


\begin{theorem} \label{dspr} 
Let $R$ be an infinite ring and assume
that, for each $m\in\ben$, $m\cdot R$ is central* in $R$.
Let $\alpha\in\ben$
and, for each $n \geq 2$ and $1 \leq i \leq \alpha$, let $d_{n,i}$
be in $R$.  Then System $(*)$ is strongly centrally partition regular over $R$.
\end{theorem}

\begin{proof} 
We already know that System~$(*)$ is centrally partition regular, by Theorem~\ref{dpr}.  We will apply Theorem~\ref{dpr} to a different central set ($C^\star$, defined below), then use that solution to build a solution in $C$ with all values of the variables distinct.  This will be possible because at each stage we will only have finitely many previously used values to avoid: since the minimal idempotent $p$ witnessing the fact that various sets $X$ are central is non-principal, sets obtained from $X$ by deleting finitely many elements remain in $p$.

So let $C$ be a central subset of $G$ and pick a minimal idempotent $p\in\beta G$ 
such that $C\in p$. As in the proof of Theorem~\ref{dpr}, $p$ is non-principal and we can assume
that $0\notin C$.  
For each $B\in p$, let $B^\star=\{x\in B:B-x\in p\}$.  If
$B\in p$ and $x\in B^\star$, then by \hslemmas{star}, $B^\star-x\in p$.

Again by \hslemmas{star}, we have that $C^\star$ is central, so pick by Theorem~\ref{dpr} 
a solution 
\[
y_1,y_2,\ldots,y_\alpha,x_{2,1},x_{2,2},z_2,x_{3,1},x_{3,2},x_{3,3},z_3,\ldots
\]
 of System $(*)$
contained in $C^\star$ such that the values of the $y_i$ are distinct.  We will use this solution to build a new solution in variables $y_i$, $u_{i,j}$ and $v_i$ for which the values taken by the variables are all distinct.

Suppose that, for $2 \leq i < n$ and $1 \leq j \leq i$, we have already chosen $u_{i,j}$ and $v_i$ distinct from 
each other and from $y_1,y_2,\ldots,y_\alpha$ such that
\[
u_{i,1} + \cdots + u_{i,i} + d_{i,1} y_1 + \cdots + d_{i,\alpha} y_\alpha = v_i.
\]
We will choose $w_1, \ldots, w_n$ in such a way that, setting $u_{n,i} = x_{n,i} + w_i$ and $v_n = z_n + w_1 + \cdots + w_n$, the same is true with $n$ replaced by $n+1$.

Let
\begin{align*}
A & = (C - x_{n,1}) \cap (C-x_{n,2}) \cap \cdots \cap (C-x_{n,n}) \cap (C-z_n), \\
B & = \{y_1, \ldots y_\alpha\} \cup \{u_{i,j} : 2 \leq i < n, 1 \leq j \leq i\} \\
  & \qquad\qquad \cup \{v_i : 2 \leq i < n\} \cup \{x_{n,i} - z_n : 1 \leq i \leq n\}, \text{ and} \\
D & = A \setminus (B \cup (B-x_{n,1}) \cup \cdots \cup (B-x_{n,n}) \cup (B-z_n) ).
\end{align*}
Since the $x_{i,j}$ and $z_i$ are in $C^\star$, $A \in p$.  Since $B$ is finite, $D \in p$. Choose $w_1\in D^\star$.

Let $2\leq k\leq n$ and suppose that we have already chosen $w_1, \ldots, w_{k-1}$ such that
\begin{enumerate}[(i)]
	\item if $\emptyset \neq F \subseteq \nhat{k-1}$, then $\sum_{j \in F} w_j \in D^\star$, and
	\item if $1 \leq i < j \leq k-1$, then $x_{n,i} + w_i \neq x_{n,j} + w_j$.
\end{enumerate}
Choose 
\[
w_k \in D^\star \cap \bigcap_{\emptyset \neq F \subseteq \nhat {k-1}} (D^\star-\sum_{j \in F} w_j) \setminus \{x_{n,j} + w_j - x_{n,k} : 1 \leq j < k\}.
\]
Then (i) and (ii) hold with $k$ replaced by $k+1$.

Having chosen $w_1,\ldots, w_n$, let $u_{n,i} = x_{n,i} + w_i$ and let $v_n = z_n + w_1 + \cdots + w_n$.  
By (i), $w_1, \ldots, w_n$ and $w_1 + \cdots + w_n$ are each in $D^\star \subseteq D$.  
Hence by the definition of $A$, $u_{n,1}, \ldots, u_{n,n}$ and $v_n$ are all in $C$, and by 
the definitions of $B$ and $D$, $u_{n,1}, \ldots, u_{n,n}$ and $v_n$ are all distinct from the 
$y_i$ ($1 \leq i \leq \alpha$), $u_{i,j}$ ($2 \leq i < n$ and $1 \leq j \leq i$) and $v_i$ ($2 \leq i < n$).  
By (ii), the $u_{n,j}$ are all distinct.  Finally, suppose that $v_n = u_{n,j}$ for 
some $j$.  Then $w_1 + \cdots + w_{j-1} + w_{j+1} + \cdots + w_n = x_{n,j}-z_n \in B$, 
but by (i) $w_1 + \cdots + w_{j-1} + w_{j+1} + \cdots + w_n \in D^\star \subseteq D$, which is a contradiction.
\end{proof}


\section{Applications} \label{applications}

In this section we show that for any two subrings $R$ and $S$ of $\beq$
such that $R$ is not contained in $S$, there is a system that is partition
regular over $R$ but not over $S$. In fact, we shall obtain this for 
a choice of the sequence $\langle d_{n,1}\rangle_{n=1}^\infty$, making 
System $(*)$ 
strongly centrally partition regular over $R$ while it has no
solutions in $S$.

In particular, this will give us a chain 
of $\gc$ subgroups of $\beq$ (where $\gc$ is the cardinality of the continuum),
any two of which 
have different partition regular systems, as stated in the introduction. (To 
see that any countable set has a chain of subsets ordered by $\ber$, simply
consider $\big\{\{x\in\beq:x<y\}:y\in\ber\big\}$.)

\begin{definition}\label{defgf} Let $P$ be the set of primes and
let $F\subseteq P$.  Then
\[
\beg_F=\left\{a/b:a\in\bez, b\in\ben\text{ and all prime factors of }
b\text{ are in }F\right\}.
\]
\end{definition}

Thus $\beg_\emptyset = \bez$, $\beg_{\{2\}}=\bed$ and $\beg_P=\beq$. It is easy to check that the 
$\beg_F$ are precisely the subrings of $\beq$. (Given a subring $R$ of $\beq$, let
$F=\{p\in P:\frac{1}{p}\in R\}$ and use the fact that $1\in R$.)

We will invoke Theorem \ref{dspr}, so we need to know that
for any subset $F$ of $P$ and any $m\in\ben$,
$m\cdot\beg_F$ is central* in $\beg_F$.  We will in fact
show that it is IP*. Recall that this means that, given
any sequence $\langle x_n\rangle_{n=1}^\infty$ in $\beg_F$,
there is some $H\in\pf(\ben)$ such that $\sum_{n\in H}x_n\in m\cdot\beg_F$ or,
equivalently, that $m\cdot \beg_F$ is a member of every idempotent in 
$\beta\beg_F$.

\begin{proposition} \label{mgipst}
Let $m \in \ben$, $F \subseteq P$ and $\langle x_n \rangle_{n=1}^{(m-1)^2+1}$ be a sequence of elements of $\beg_F$.  Then there exists $\emptyset \neq H \subseteq \nhat{(m-1)^2+1}$ such that $\sum_{n \in H} x_n \in m \cdot \beg_F$.
\end{proposition}

\begin{proof}
Write the $x_n$ over a common denominator: choose $s \in \ben$ with all prime factors in $F$ such that $x_n = y_n / s$ with $y_n \in \bez$.  At least $m$ of the $y_n$ must have the same residue modulo $m$; let $H$ be a set of size $m$ such that $y_n \equiv h \pmod m$ for $n \in H$.  Then $\sum_{n \in H} y_n = km$ for some $k$, hence $\sum_{n \in H} x_n = km/s \in m \cdot \beg_F$.
\end{proof}

\begin{theorem}\label{onenotother} Let $F$ and $H$ be subsets of $P$ with $H\setminus F\neq \emp$ and 
pick $q\in H\setminus F$.  Let $\alpha=1$ and for $k\in\ben$, let $d_{n,1}= \frac 1 {q^n}$.
Then System $(*)$ is strongly centrally partition regular over $\beg_H$ but is not partition
regular over $\beg_F$.\end{theorem}

\begin{proof} It is immediate that System $(*)$ has no solutions in $\beg_F$.  
By Theorem \ref{dspr} with $R=\beg_H$, System $(*)$ is strongly 
centrally partition regular over $\beg_H$.\end{proof}

By applying this to a chain of size $\gc$ of subsets of the primes, we 
immediately obtain a chain of $\gc$ subrings of $\beq$, no two of which have
the same partition regular systems.

If we want to separate $\beq$ from all proper subrings simultaneously then we
have the following, whose proof is identical.
Let $p_1, p_2, \ldots$ be an enumeration of the primes.

\begin{theorem} \label{overq} 
Let $\alpha=1$ and for $n\in\ben$, let $d_{n,1}=\prod_{t=1}^n\frac{1}{p_t^n}$.  
Then System $(*)$ is
strongly centrally partition regular over $\beq$, but is not partition regular over $\beg_F$ for any proper 
subset $F$ of $P$. \qed
 \end{theorem}

One might raise the objection that it almost seems like cheating to show that a system
is not partition regular over $G$ by showing that it has no solutions there at all.
We see now that by taking $\alpha=2$, we can get 
examples where System $(*)$ has
solutions in $\ben$, but the conclusions of Theorems \ref{onenotother} and \ref{overq} still hold.

\begin{theorem} \label{onenototherplus} 
Let $F$ and $H$ be subsets of $P$  
with $H\setminus F\neq \emp$ and 
pick $q\in H\setminus F$.  Let $\alpha=2$ and, for $n\in\ben$, let $d_{n,1}=\frac{-1}{q^n}$
and $d_{n,2}=\frac{2}{q^n}$. Then System $(*)$ has solutions in $\ben$ and
is strongly centrally partition regular over $\beg_H$, but is not partition
regular over $\beg_F$.\end{theorem}

\begin{proof}   
By Theorem \ref{dspr} with $R=\beg_H$, System $(*)$ is strongly 
centrally partition regular over $\beg_H$.
Let $y_1=2$ and $y_2=1$. Then for every $n\in\ben$, $d_{n,1}y_1+d_{n,2}y_2=0$ so
it is easy to find a solution to System $(*)$ in $\ben$.

To see that  System $(*)$ is not partition regular over $\beg_F$, 
two-colour $\beg_F\setminus\{0\}$ so that for all $x\in \beg_F\setminus\{0\}$,
$x$ and $2x$ do not have the same colour.  (For example colour by the
parity of $\lfloor \log_2(|x|)\rfloor$.)  Suppose we have
a monochromatic solution to System $(*)$ in $\beg_F$.
We have that $y_1 \neq 2y_2$ and, for each $n \in \ben$, 
$(2y_2-y_1)/q^n = z_n - x_{n,1} - \cdots - x_{n,n} \in \beg_F$, 
which is a contradiction for $n$ sufficiently large.
\end{proof}

Similarly, we have an analogue of Theorem \ref{overq}.

\begin{theorem}\label{overqplus} 
Let $\alpha=2$ and for $n\in\ben$, let $d_{n,1}=\prod_{t=1}^n\frac{-1}{p_t^n}$ and
$d_{n,2}=\prod_{t=1}^n\frac{2}{p_t^n}$.  
Then System $(*)$ is
strongly partition regular over $\beq$ and has solutions in $\ben$, but 
is not partition regular over $\beg_F$ for any proper nonempty
subset $F$ of $P$. \qed \end{theorem}

Let us end by remarking that it would be interesting to understand
what happens beyond $\beq$---in other words, for subrings (or subgroups) that
lie between $\beq$ and $\ber$. Of course, if one allows non-rational
coefficients then it is easy to separate sets, so the interest would be for
systems of equations whose coefficients are integers (or, equivalently,
rationals). 

We see now that the system mentioned in the Introduction that 
distinguishes $\ber$ from $\beq$ in fact distinguishes any uncountable subgroup
$G$ of $\ber$ from $\beq$.

In the following result we use, as in \cite{HLS}, the Baumgartner--Hajnal theorem \cite[Theorem 1]{BH}.
This theorem states that
if $A$ is a linearly ordered set with
the property that whenever
$\varphi :A\to\ben$, there is an infinite
increasing sequence in $A$ on which $\varphi$ is constant, then
for any countable ordinal $\alpha$, and
any finite colouring $\psi$ of the two-element subsets of $A$ there is a subset $B$ of $A$ which
has order type $\alpha$ such that $\psi$ is constant on
the two-element subsets of $B$.  (The theorem was proved in \cite{BH} using Martin's axiom followed
by an absoluteness argument to show that it is a theorem of ZFC. A direct combinatorial proof
was obtained by Galvin in \cite[Theorem 4]{G}.)

\begin{theorem}\label{GfromQ} Let $G$ be an uncountable subgroup of $\ber$.
Then the system of equations $y_n=x_n-x_{n+1}$ $(n=0,1,2,\ldots)$ is partition regular
over $G$ but not over $\beq$.\end{theorem}

\begin{proof} It was shown in \cite{HLS} (immediately before Question 6) that
the system is not partition regular over $\beq$.  To show that the system
is partition regular over $G$, we use the Baumgartner--Hajnal theorem.
For this we need to observe that given any countable colouring of $G$,
there is a monochromatic increasing sequence.  To see this, let $\varphi:G\to\ben$ and
define $\psi:G\to\ben\times\ben$ by $\psi(x)=\big(\varphi(x),\varphi(-x)\big)$.
Pick $(n,m)\in\ben\times\ben$ such that $A=\psi^{-1}[\{(n,m)\}]$ is infinite.
Then $A$ contains a sequence $\langle x_t\rangle_{t=1}^\infty$ which is
either increasing or decreasing. If $\langle x_t\rangle_{t=1}^\infty$ is
increasing, then it is an increasing sequence in $\varphi^{-1}[\{n\}]$.
If $\langle x_t\rangle_{t=1}^\infty$ is
decreasing, then $\langle -x_t\rangle_{t=1}^\infty$ is an increasing sequence in $\varphi^{-1}[\{m\}]$.

Now let $G$ be finitely coloured by $\varphi$ and, given a two-element subset $\{x,y\}$ of $G$,
define $\psi(\{x,y\})=\varphi(|x-y|)$. By the Baumgartner--Hajnal theorem, pick
an increasing sequence $\langle z_\sigma\rangle_{\sigma<\omega+1}$ such
that $\psi$ is constant on $\big\{\{z_\sigma,z_\tau\}:\sigma<\tau\big\}$.  Given
$n\in\ben$, let $x_n=z_\omega-z_n$ and let $y_n=z_{n+1}-z_n$.\end{proof}

Perhaps even more interesting would be to understand what happens for
subgroups of $\beq$. The following is the obvious question to ask.

\begin{question}\label{othersubgroups} If $G$ and $H$ are subgroups of $\beq$
such that $G$ does not contain a subgroup isomorphic to $H$, must there exist
a system (of linear equations with integer coefficients) that is partition regular over $H$ but not over $G$?
\end{question}

It is easy to check that every subgroup of $\beq$ that contains $1$ is the set of rationals $a/b$ such that the multiplicity of $p_i$ in the prime factorisation of $b$ is at most $k_i$, where each $k_i$ is either a non-negative integer or $\infty$.  Given two such sequences $k$ and $k'$, if there is some $i$ for which $k_i = \infty$ while $k'_i$ is finite, then the corresponding groups can be separated by the methods of this section.  But if for every $i$, either both $k_i$ and $k'_i$ are infinite, or both are finite, then we are unable to say anything.

The most attractive special case is surely the following.

\begin{question}\label{squarefree} Does there exist a system (of linear equations
with integer coefficients) that is partition
regular over the set of rationals with squarefree denominators but is not
partition regular over the integers?
\end{question}


\begin{thebibliography}{9}

\bibitem{BHL} B. Barber, N. Hindman, and I. Leader, {\it Partition regularity in the rationals\/}, 
J. Comb.\ Theory (Series A) {\bf 120} (2013), 1590--1599.

\bibitem{BHLS} B. Barber, N. Hindman, I. Leader, and D. Strauss, {\it Partition regularity without the columns property \/}, 
Proc.\ Amer.\ Math.\ Soc., to appear\footnote{Currently available at 
{\tt http://nhindman.us/preprint.html}.}.

\bibitem{BH} J.\ Baumgartner and A.\ Hajnal, {\it A proof
(involving Martin's Axiom) of a partition relation\/},
Fund.\ Math.\ {\bf 78} (1973), 193--203.

\bibitem{G} F.\ Galvin, {\it On a partition theorem of
Baumgartner and Hajnal\/}, Colloq.\ Math.\ Soc.\ J\'anos Bolyai,
Vol.\ 10, North Holland, Amsterdam, 1975, 711-729.

\bibitem{H} N. Hindman, {\it Partition regularity of matrices\/}, 
Integers {\bf 7(2)} (2007), A-18.
{\tt http://www.integers-ejcnt.org/vol7-2.html} 

\bibitem{HLS} N. Hindman, I. Leader, and D. Strauss,
{\it Open problems in partition regularity\/},
Combin.\ Probab.\ Comput.\ {\bf 12} (2003), 571--583.

\bibitem{HS} N. Hindman and D. Strauss,
{\it Algebra in the Stone--\v Cech compactification: theory and applications, 2nd edition\/},
Walter de Gruyter \& Co., Berlin, 2012.

\bibitem{R} R. Rado, {\it Studien zur Kombinatorik\/},
Math.\ Z. {\bf 36} (1933), 424--470.

\bibitem{Rb} R.\ Rado, {\it Note on combinatorial analysis\/}, Proc.\
London Math.\ Soc.\ {\bf 48} (1943), 122--160.



\end{thebibliography}
\end{document}